\documentclass{article}

\usepackage{fancyhdr}
\usepackage{setspace, amsmath, amsthm, amssymb, amsfonts, amscd, epic, graphicx, ulem, dsfont}
\usepackage{amsthm}
\usepackage{arxiv}

\usepackage[utf8]{inputenc} 
\usepackage[T1]{fontenc}    
\usepackage{hyperref}       
\usepackage{url}            
\usepackage{booktabs}       
\usepackage{amsfonts}       
\usepackage{nicefrac}       
\usepackage{microtype}      
\usepackage{lipsum}
\usepackage{graphicx}
\graphicspath{ {./images/} }
\newtheorem{theorem}{Theorem}

\newtheorem{proposition}[theorem]{Proposition}
\newtheorem{corollary}[theorem]{Corollary}
\newtheorem{example}[theorem]{Example}

\newtheorem{remark}[theorem]{Remark}

\title{On the generalized of $p$-harmonic maps}

\author{Bouchra Merdji \\
Mascara University, Faculty of Exact Sciences, Mascara 29000, Algeria\\
\texttt{bouchra.merdji@univ-mascara.dz}
\And Ahmed Mohammed Cherif\\
Mascara University, Faculty of Exact Sciences, Mascara 29000, Algeria\\
\texttt{a.mohammedcherif@univ-mascara.dz}}

\begin{document}

\maketitle

\begin{abstract}
In this paper, we extend the definition of $p$-harmonic and $p$-biharmonic maps between
Riemannian manifolds. We present some new properties for the generalized stable $p$-harmonic maps.\\
 \textit{ keywords:} $p$-Harmonic maps, $p$-Biharmonic maps.
\end{abstract}

\section{Introduction}
Consider a smooth map $\varphi:(M,g)\longrightarrow (N,h)$ between Riemannian manifolds, and let $p$ be a smooth positive function on $M$ such that $p(x)\geq2$
for all $x\in M$. For any compact domain $D$ of $M$ the $p(\cdot)$-energy functional of $\varphi$ is defined by
\begin{equation}\label{eq1.1}
E_{p(\cdot)}(\varphi;D)=\int_{D}\frac{|d\varphi|^{p(x)}}{p(x)}v_{g},
\end{equation}
where $|d\varphi|$ is the Hilbert-Schmidt norm of the differential $d\varphi$ and $v^g$ is the volume element on $(M,g)$.
A map is called $p(\cdot)$-harmonic if it is a critical point of the $p(\cdot)$-energy functional over any compact subset $D$ of $M$.
$p(\cdot)$-harmonic maps is a natural generalization of harmonic map (\cite{baird, ES}) and $p$-harmonic map (\cite{BG,BI,ali}).
We denote by
\begin{equation}\label{eq1.2}
    \tau(\varphi)=\operatorname{trace}_g\nabla d\varphi
    =\sum_{i=1}^m\big\{\nabla^\varphi _{e_i}d\varphi(e_i)-d\varphi(\nabla^M _{e_i} e_i)\big\}.
\end{equation}
the tension field of $\varphi$, where $\{e_i\}_{i=1}^{m}$ is an orthonormal frame on $(M,g)$, $\nabla^{M}$ is the Levi-Civita connection of $(M,g)$,
and $\nabla^{\varphi}$ denote the pull-back connection on $\varphi^{-1}TN$.\\
In this paper, we investigate some properties for $p(\cdot)$-harmonic maps between two Riemannian manifolds. In particular,
we present the first and the second variation of the $p(\cdot)$-energy. We also extend the definition of $p$-biharmonic maps between two
Riemannian manifolds (\cite{cherif2}).

\section{$p(\cdot)$-Harmonic Maps}

\begin{theorem}[The first variation of the $p(\cdot)$-energy]\label{th1}
Let $\varphi:(M,g)\rightarrow (N,h)$ be a smooth map between two Riemannian manifolds,
$\{\varphi_{t}\}_{t\in (-\epsilon,\epsilon)}$ a smooth variation of $\varphi$ supported in compact domain $D$ of $M$. Then
\begin{equation}\label{eq2.1}
    \frac{d}{dt}E_{p(\cdot)}(\varphi_{t};D)\Big|_{t=0}=-\int_{D}h(v,\tau_{p(\cdot)}(\varphi))v_{g},
\end{equation}
where $\tau_{p(\cdot)}(\varphi)$ denotes the $p(\cdot)$-tension field of $\varphi$ given by
\begin{align}\label{eq2.2}
 \tau_{p(\cdot)}(\varphi)&=\operatorname{trace}_g\nabla|d\varphi|^{p(x)-2} d\varphi,
 \end{align}
and $v=\frac{d\varphi_{t}}{dt}\big|_{t=0}$ denotes the variation vector field of $\{\varphi_{t}\}_{t\in (-\epsilon,\epsilon)}$.
\end{theorem}

\begin{proof}
  Let $\phi :M\times (-\epsilon,\epsilon)\longrightarrow N$ be a smooth map defined by
  $$\phi(x,t)=\varphi_{t}(x),\quad\forall (x,t)\in M\times (-\epsilon,\epsilon).$$
  We have $\phi (x,0)=\varphi(x)$ for all $x\in M$, and the variation vector field associated to the variation $\{\varphi_{t}\}_{t\in(-\epsilon,\epsilon)}$ is given by
  $$v(x)=d_{(x,0)}\phi(\frac{\partial}{\partial t})\Big|_{t=0},\quad\forall x\in M.$$
  Let $\{e_i\}_{i=1}^{m}$ be an orthonormal frame on $(M,g)$. We compute
  \begin{align}
  \frac{d}{dt}E_{p(\cdot)}(\varphi_{t};D)\Big|_{t=0}
   &= \nonumber\frac{d}{dt}\int_{D}\frac{|d\varphi_{t}|^{p(x)}}{p(x)}v_{g}\bigg|_{t=0}\\ \nonumber
   &=\int_{D}\frac{1}{p(x)}\frac{\partial}{\partial t}|d\varphi_{t}|^{p(x)}\bigg|_{t=0}v_{g} \\
   &=\nonumber\int_{D}\frac{1}{p(x)}\frac{\partial}{\partial t}(|d\varphi_{t}|^{2})^{\frac{p(x)}{2}}\bigg|_{t=0}v_{g}\\
   &=\nonumber\sum_{i=1}^m\int_{D}\frac{1}{p(x)}\frac{\partial }{\partial t}h(d\varphi_{t}(e_i),d\varphi_{t}(e_i))^{\frac{p(x)}{2}}\bigg|_ {t=0}v_{g}\\
   &=  \nonumber\sum_{i=1}^m\int_{D}\frac{1}{p(x)}\frac{\partial}{\partial t}h(d\phi(e_i,0),d\phi(e_i,0))^{\frac{p(x)}{2}}\bigg|_{t=0}v_g\\
   &=\label{eq2.3} \sum_{i=1}^m\int_{D}h(\nabla _{\frac{\partial}{\partial t}}^{\phi}d\phi(e_i,0),d\phi(e_i,0))(|d\varphi_{t}|^{2})^{\frac{p(x)}{2}-1}\bigg|_{t=0} v_{g}.
   \end{align}
  By using the property $$\nabla^\phi_X d\phi(Y)=\nabla^\phi_Y d\phi(X)+d\phi([X,Y]),$$
  with $X=\frac{\partial}{\partial t}$, $Y=(e_i,0)$, and $[\frac{\partial}{\partial t},(e_i,0)]=0$, the equation (\ref{eq2.3}) becomes
  \begin{align}\label{eq2.4}
     \frac{d}{dt}E_{p(\cdot)}(\varphi_{t};D)\bigg|_{t=0}
     &= \nonumber\sum_{i=1}^m\int_{D}h(\nabla^{\phi}_{(e_i,0)}d\phi(\frac{\partial}{\partial t}),d\phi(e_i,0))|d\varphi_{t}|^{p(x)-2}\bigg|_{t=0}v_{g} \\
     &=\nonumber\sum_{i=1}^m\int_{D}h(\nabla^{\varphi}_{e_i}v,|d\varphi|^{p(x)-2}d\varphi(e_i))v_g\\
     &=\nonumber\sum_{i=1}^m\int_{D}\Big[e_i h(v,|d\varphi|^{p(x)-2}d\varphi(e_i))\\
     &\quad-h(v,\nabla^{\varphi}_{e_i}|d\varphi|^{p(x)-2}d\varphi(e_i))\Big]v_g.
  \end{align}
Let $\omega \in\Gamma(T^{*}M)$ defined by
\begin{equation*}
\omega (X)=h(v,|d\varphi|^{p(x)-2}d\varphi(X)),\quad \forall X\in\Gamma(TM)
\end{equation*}
The divergence of $\omega$ is given by
\begin{equation}\label{eq2.5}
\operatorname{div}^{M}\omega=\sum_{i=1}^m\left[e_ih(v,|d\varphi|^{p(x)-2}d\varphi(e_i))-h(v,|d\varphi|^{p(x)-2}d\varphi(\nabla^{M} _{e_i} e_i))\right].
\end{equation}
By equations (\ref{eq2.4}), (\ref{eq2.5}), and the divergence Theorem \cite{baird}, we get
\begin{align}\label{eq2.6}
   \frac{d}{dt}E_{p(\cdot)}(\varphi_{t};D)\bigg|_{t=0} =& \nonumber
    \sum_{i=1}^m\int_{D} h\big(v,|d\varphi|^{p(x)-2}d\varphi(\nabla ^{M} _{e_i} e_i)
    -\nabla^{\varphi}_{e_i}|d\varphi|^{p(x)-2}d\varphi(e_i)\big)v_{g}\\
=&-\sum_{i=1}^m\int_{D} h\left(v,\left[\nabla^{}_{e_i}|d\varphi|^{p(x)-2}d\varphi\right](e_i)\right)v_{g}.
\end{align}
\end{proof}

\begin{corollary}
A smooth map $\varphi:(M,g)\longrightarrow (N,h)$ between two Riemannian manifolds is $p(\cdot)$-harmonic if and only if
$$\tau_{p(\cdot)}(\varphi)=|d\varphi|^{p(x)-2}\tau(\varphi)+d\varphi(\operatorname{grad}^M|d\varphi|^{p(x)-2})=0.$$
\end{corollary}

\begin{example}
The restriction of inversion $$\varphi:\mathbb{R}^n\backslash\{0\}\longrightarrow\mathbb{R}^n\backslash\{0\},\quad x\longmapsto\frac{x}{\|x\|^2},$$
to $M=\{x\in\mathbb{R}^n\backslash\{0\},\,\|x\|^2>\sqrt{n}\}$ is $p(\cdot)$-harmonic map, where the function $p$ is given by
$$p(x)=n+\frac{c}{2\ln(\|x\|^2)-\ln(n)},\quad\forall x\in M,$$
for some constant $c\geq0$. Here, $|d\varphi|(x)=\frac{\sqrt{n}}{\|x\|^2}$ for all $x\in\mathbb{R}^n\backslash\{0\}$.
\end{example}

\begin{example}
Let $F:\mathbb{R}\longrightarrow[2,\infty)$ be a smooth function. The map $$\varphi:\mathbb{R}^n\backslash\{0\}\longrightarrow\mathbb{S}^n,\quad x\longmapsto\frac{x}{\|x\|},$$ is $p(\cdot)$-harmonic, where $p(x)=F(\|x\|^2)$ for all $x\in\mathbb{R}^n\backslash\{0\}$. The Hilbert-Schmidt norm of  $d\varphi$
is given by $|d\varphi|(x)=\frac{\sqrt{n-1}}{\|x\|}$ for all $x\in\mathbb{R}^n\backslash\{0\}$.
\end{example}

\begin{remark}
A smooth harmonic map, i.e., $\tau(\varphi)=0$, with constant energy density $\frac{|d\varphi|^2}{2}$ is not always $p(\cdot)$-harmonic.
The previous examples prove the following results;
There is no equivalence between the $p(\cdot)$-harmonicity and the harmonicity of a smooth map $\varphi:(M,g)\longrightarrow (N,h)$.
There are $p(\cdot)$-harmonic maps which have non-constant Hilbert-Schmidt norm and they are not harmonic.
\end{remark}

\section{Stable $p(\cdot)$-Harmonic Maps}
\begin{theorem}[The second variation of the $p(\cdot)$-energy]
Let $\varphi$ be a smooth $p(\cdot)$-harmonic map between two Riemannian manifolds $(M,g)$ and $(N,h)$. Then we have
\begin{equation}\label{eq3.1}
\frac{\partial ^{2}}{\partial t \partial s}E_{p(\cdot)}(\varphi_{t,s};D)\bigg|_{t=s=0}=\int_{D}h(J^{\varphi}_{p(\cdot)}(v),w)v_g,
\end{equation}
where $\{\varphi_{t,s} \}_{(t,s)\in (-\epsilon,\epsilon)\times (-\epsilon,\epsilon)}$ is a smooth variation supported in compact domain $D \subset M$ of $\varphi$,
\begin{equation}\label{eq3.2}
 v=\frac{\partial \varphi_{t,s}}{\partial t}\bigg|_{t=s=0}, \quad w=\frac{\partial \varphi_{t,s}}{\partial s}\bigg|_{t=s=0},
\end{equation}
and $J_{p(\cdot)}^{\varphi}$ the generalized Jacobi operator of $\varphi$ given by
\begin{align}\label{eq3.3}
J_{p(\cdot)}^{\varphi}(v)
=&\nonumber-| d\varphi|^{p(x)-2} \operatorname{trace}_{g}R^{N}(v,d\varphi)d\varphi
-\operatorname{trace}_{g}\nabla^{\varphi}| d\varphi| ^{p(x)-2}\nabla^{\varphi}v\\
&-\operatorname{trace}_g \nabla (p(x)-2) |d\varphi |^{p(x)-4} \langle\nabla^{\varphi}v,d\varphi\rangle d\varphi.
\end{align}
Here $\langle\,,\,\rangle$ denote the inner product on $T^{*}M \otimes \varphi^{-1}TN$.
\end{theorem}
\begin{proof}
Let $\phi :M\times (-\epsilon,\epsilon)\times(-\epsilon,\epsilon)\longrightarrow N$ be a smooth map defined by $\phi(x,t,s)=\varphi_{t,s}(x)$. We have $\phi(x,0,0)=\varphi(x)$, and the variation vectors fields associated to the variation $\{\varphi_{t,s}\}_{(t,s)\in(-\epsilon,\epsilon)\times(-\epsilon,\epsilon)}$ are given by
	\begin{equation}\label{eq3.4}
	v(x)=d_{(x,0,0)}\phi(\frac{\partial}{\partial t}), \quad w(x)=d_{(x,0,0)}\phi(\frac{\partial}{\partial s}), \quad \forall x \in M.
	\end{equation}
Let $\{e_i\}_{i=1}^{m}$ be an orthonormal frame with respect to $g$ on $M$ such that $\nabla ^{M}_{e_i}e_j=0$ at $x\in M$ for all $i,j=1,\dots,m $. We compute
\begin{eqnarray}\label{eq3.5}
\frac{\partial ^{2}}{\partial t\partial s}E_{p(\cdot)}(\varphi_{t,s};D)\bigg|_{t=s=0}
&=&\nonumber\frac{\partial^{2}}{\partial t\partial s}\int_{D}\frac{|d\varphi_{t,s}|^{p(x)}}{p(x)} v_g \bigg|_{t=s=0}\\
&=&\int_{D}\frac{1}{p(x)} \frac{\partial^{2}}{\partial t\partial s}|d\varphi_{t,s}|^{p(x)}\bigg|_{t=s=0}v_g .
\end{eqnarray}	
First, note that
\begin{align*}
\frac{1}{p(x)}\frac{\partial ^{2}}{\partial t \partial s}|d\varphi_{t,s} |^{p(x)}
&= \frac{1}{p(x)}\frac{\partial }{\partial t}\left( \frac{\partial}{\partial s} (|d\varphi _ {t,s}|^{2})^{\frac{p(x)}{2}}\right)\\
&=\frac{1}{2}\frac{\partial}{\partial t}\left( (|d\varphi_{t,s}|^{2})^{\frac{p(x)}{2}-1}\frac{\partial}{\partial s}|d\varphi_{t,s}|^{2}\right)\\
&=\sum_{i=1}^m\frac{\partial}{\partial t}\left( (|d\varphi_{t,s}|^{2})^{\frac{p(x)}{2}-1}h(\nabla^\phi_{\frac{\partial}{\partial s}}d\phi(e_i,0,0),d\phi(e_i,0,0))\right).
\end{align*}
Thus
	\begin{align*}
	\frac{1}{p(x)}\frac{\partial ^{2}}{\partial t \partial s}|d\varphi_{t,s} |^{p(x)}
&=\sum_{i=1}^m\frac{\partial}{\partial t}\left(|d\varphi_{t,s}|^{2} \right)^{\frac{p(x)}{2}-1}h(\nabla^\phi_{\frac{\partial}{\partial s}}d\phi(e_i,0,0),d\phi(e_i,0,0))\\
&+\sum_{i=1}^m|d\varphi_{t,s}|^{p(x)-2}h(\nabla^\phi_{\frac{\partial }{\partial t}}\nabla^\phi _{\frac{\partial}{\partial s}}d\phi(e_i,0,0),d\phi(e_i,0,0))\\
&+\sum_{i=1}^m|d\varphi_{t,s}|^{p(x)-2}h(\nabla^{\phi}_{\frac{\partial}{\partial s}}d\phi(e_i,0,0),\nabla^{\phi}_{\frac{\partial}{\partial t}}d\phi(e_i,0,0)).
\end{align*}
So that
\begin{align*}
\frac{1}{p(x)}\frac{\partial ^{2}}{\partial t \partial s}|d\varphi_{t,s} |^{p(x)}	
&=\sum_{i,j=1}^m(p(x)-2)|d\varphi_{t,s}|^{p(x)-4}h(\nabla^{\phi}_{\frac{\partial}{\partial t}}d\phi(e_j,0,0),d\phi(e_j,0,0))\\
    &\quad h(\nabla^\phi_{\frac{\partial}{\partial s}}d\phi(e_i,0,0),d\phi(e_i,0,0))\\
	&+\sum_{i=1}^m|d\varphi_{t,s}|^{p(x)-2}h(\nabla^\phi_{\frac{\partial }{\partial t}}\nabla^\phi _{\frac{\partial}{\partial s}}d\phi(e_i,0,0),d\phi(e_i,0,0))\\
	&+\sum_{i=1}^m|d\varphi_{t,s}|^{p(x)-2}h(\nabla^{\phi}_{\frac{\partial}{\partial s}}d\phi(e_i,0,0),\nabla^{\phi}_{\frac{\partial}{\partial t}}d\phi(e_i,0,0)).
\end{align*}
 By the definition of the curvature tensor of $(N,h)$ and the properties
  $$\nabla^{\phi}_{\frac{\partial }{\partial t}}d\phi(e_i,0,0)=\nabla^{\phi}_{(e_i,0,0)}d\phi(\frac{\partial}{\partial t}),\quad
  \nabla^{\phi}_{\frac{\partial }{\partial s}}d\phi(e_i,0,0)=\nabla^{\phi}_{(e_i,0,0)}d\phi(\frac{\partial}{\partial s}),$$
 with  $\left[ \frac{\partial}{\partial t},(e_i,0,0)\right]=0$, we obtain the following equation
\begin{align}\label{eq3.7}
	\frac{1}{p(x)}\frac{\partial ^{2}}{\partial t \partial s}|d\varphi_{t,s} |^{p(x)}\bigg|_{t=s=0}
    &\nonumber= \sum_{i=1}^m h\left(\nabla^{\varphi}_{e_i}w,(p(x)-2)|d\varphi|^{p(x)-4}\langle\nabla^{\varphi}v,d\varphi\rangle d\varphi(e_i)\right)\\
	&\nonumber\quad-|d\varphi|^{p(x)-2}\sum_{i=1}^mh({R}^{N}(v,d\varphi(e_i))d\varphi(e_i),w)\\
	&\nonumber\quad+\sum_{i=1}^mh\left(\nabla^{\varphi}_{e_i}\nabla_{\frac{\partial}{\partial t}}^{\phi}d\phi(\frac{\partial}{\partial s})\Big|_{t=s=0},|d\varphi|^{p(x)-2}d\varphi(e_i)\right)\\
	&\quad+\sum_{i=1}^mh\left(\nabla_{e_i}^{\varphi}w,|d\varphi|^{p(x)-2}\nabla^{\varphi}_{e_i}v\right).
\end{align}
Let $\omega_1,\omega_2,\omega_3\in \Gamma(T^*M)$ defined by
\begin{eqnarray*}
\omega_1(X)   &=&  h\left(w,(p(x)-2)|d\varphi|^{p(x)-4}\langle\nabla^{\varphi}v,d\varphi\rangle d\varphi(X)\right);\\
\omega_2(X)   &=&  h\left(\nabla_{\frac{\partial}{\partial t}}^{\phi}d\phi(\frac{\partial}{\partial s})\Big|_{t=s=0},|d\varphi|^{p(x)-2}d\varphi(X)\right):\\
\omega_3(X)   &=&  h\left(w,|d\varphi|^{p(x)-2}\nabla^{\varphi}_{X}v\right),\quad\forall X\in\Gamma(TM).
\end{eqnarray*}
The divergence of $\omega_1,$ $\omega_2,$ and $\omega_3$ are given by
\begin{eqnarray*}
\operatorname{div}^{M}\omega_1   &=& \sum_{i=1}^m e_ih\left(w,(p(x)-2)|d\varphi|^{p(x)-4}\langle\nabla^{\varphi}v,d\varphi\rangle d\varphi(e_i)\right);\\
\operatorname{div}^{M}\omega_2  &=& \sum_{i=1}^m e_ih\left(\nabla_{\frac{\partial}{\partial t}}^{\phi}d\phi(\frac{\partial}{\partial s})\Big|_{t=s=0},|d\varphi|^{p(x)-2}d\varphi(e_i)\right):\\
\operatorname{div}^{M}\omega_3   &=& \sum_{i=1}^m e_ih\left(w,|d\varphi|^{p(x)-2}\nabla^{\varphi}_{e_i}v\right),\quad\forall X\in\Gamma(TM).
\end{eqnarray*}
By equations (\ref{eq3.5}), (\ref{eq3.7}), the $p(\cdot)$-harmonicity condition of $\varphi$, and the divergence Theorem, we obtain
\begin{eqnarray}\label{eq3.8}
\frac{\partial ^{2}}{\partial t\partial s}E_{p(\cdot)}(\varphi_{t,s};D)\bigg|_{t=s=0}
&=&\nonumber-\int_{D}\sum_{i=1}^m h\left(w,\nabla^{\varphi}_{e_i}(p(x)-2)|d\varphi|^{p(x)-4}\langle\nabla^{\varphi}v,d\varphi\rangle d\varphi(e_i)\right)v_g\\
&&\nonumber-\int_{D}|d\varphi|^{p(x)-2}\sum_{i=1}^mh(w,{R}^{N}(v,d\varphi(e_i))d\varphi(e_i))v_g\\
&& -\int_{D}\sum_{i=1}^mh\left(w,\nabla_{e_i}^{\varphi}|d\varphi|^{p(x)-2}\nabla^{\varphi}_{e_i}v\right)v_g.
\end{eqnarray}
The proof is completed.
\end{proof}

If $(M,g)$ is a compact Riemannian manifold, $\varphi$ be a $p(\cdot)$-harmonic map from $(M,g)$ to Riemannian manifold $(N,h)$, and for any vector field $v$ along $\varphi$,
\begin{equation}\label{eq1.16}
I^{\varphi}_{p(\cdot)}(v,v)\equiv\int_{M}h(v,J_{p(\cdot)}^{\varphi}(v))\,v_{g}\geq 0,
\end{equation}
then $\varphi$ is called a stable $p(\cdot)$-harmonic map. Note that, the definition of stable $p(\cdot)$-harmonic maps is a generalization of stable harmonic maps (\cite{YX}), is also a generalization of stable $p$-harmonic maps (\cite{CL,NS}). By using the Green Theorem \cite{baird} it is easy to prove that
\begin{eqnarray}\label{eq3.9}
I^{\varphi}_{p(\cdot)}(v,v)
   &=& -\int_{M}|d\varphi|^{p(x)-2}\sum_{i=1}^mh(v,{R}^{N}(v,d\varphi(e_i))d\varphi(e_i))v_g\\
   &&\nonumber +\int_M |d\varphi|^{p(x)-2}|\nabla^\varphi v|^2 v_g
       +\int_M (p(x)-2)|d\varphi|^{p(x)-4} \langle\nabla^{\varphi}v,d\varphi\rangle^2 v_g.
\end{eqnarray}
From equation (\ref{eq3.9}), we deduce the following result.
\begin{proposition}
Every $p(\cdot)$-harmonic map from compact Riemannian manifold $(M,g)$ to Riemannian manifold $(N,h)$ has $\operatorname{Sect}^N\leq0$ is stable.
\end{proposition}

In the case where the codomain of the stable $p(\cdot)$-harmonic map is the standard sphere $\mathbb{S}^n$, we have the following result.

\begin{theorem}\label{th2:sphere}
Let $(M,g)$ be a compact Riemannian manifold. When $n > 2$,
any stable $p(\cdot)$-harmonic map $\varphi:(M,g)\longrightarrow\mathbb{S}^n$ must be constant,
where $p$ is a smooth positive function on $M$ such that $2\leq p(x)<n$ for all $x\in M$.
\end{theorem}

\begin{proof}
Choose a normal orthonormal frame $\{e_i\}_{i=1}^m$ at point $x$ in $(M,g)$.
We set $\lambda(y)=\langle\alpha,y\rangle_{\mathbb{R}^{n+1}}$, for all $y\in \mathbb{S}^n,$
where $\alpha\in\mathbb{R}^{n+1}$. Let $v=\operatorname{grad}^{\mathbb{S}^n}\lambda$.
We have $\nabla^{\mathbb{S}^n}_{X}v=-\lambda X$ for all $X\in\Gamma(T\mathbb{S}^n)$,
where $\nabla^{\mathbb{S}^n}$ is the Levi-Civita connection on $\mathbb{S}^n$ with respect to the standard
metric of the sphere (see \cite{YX}). We compute
\begin{eqnarray}\label{28}
\sum_{i=1}^m\nabla^{\varphi}_{e_{i}} |d\varphi|^{p(x)-2} \nabla^{\varphi}_{e_{i}}(v\circ\varphi )
&=& \nonumber \nabla^{\varphi}_{\operatorname{grad}^{M}|d\varphi|^{p(x)-2}}(v\circ\varphi) \\
&&+\sum_{i=1}^m |d\varphi|^{p(x)-2}\nabla^{\varphi}_{e_{i}} \nabla^{\varphi}_{e_{i}} (v\circ\varphi ).
\end{eqnarray}
By using the property $\nabla^{\mathbb{S}^{n}}_{X}v = -\lambda X$,
the first term of \eqref{28} is given by
\begin{equation}\label{29}
\nabla^{\varphi}_{\operatorname{grad}^{M}|d\varphi|^{p(x)-2}}(v\circ\varphi) = -(\lambda \circ\varphi)d\varphi(\operatorname{grad}^{M}|d\varphi|^{p(x)-2}),
\end{equation}
and the seconde term of \eqref{28} is given by
\begin{eqnarray}\label{30}
\sum_{i=1}^m|d\varphi|^{p(x)-2}\nabla^{\varphi}_{e_{i}} \nabla^{\varphi}_{e_{i}}(v\circ\varphi )
& = &-\sum_{i=1}^m |d\varphi|^{p(x)-2}\nabla^{\varphi}_{e_{i}}(\lambda \circ\varphi)d\varphi (e_{i})\nonumber \\
&= &\nonumber -\sum_{i=1}^m|d\varphi|^{p(x)-2}<d\varphi(e_{i}),v\circ\varphi>d\varphi(e_{i})\\
&& -(\lambda \circ\varphi)|d\varphi|^{p(x)-2}\tau(\varphi).
\end{eqnarray}
Substituting the formulas \eqref{29} and \eqref{30} in \eqref{28} gives
\begin{eqnarray}\label{32}
\sum_{i=1}^m\nabla^{\varphi}_{e_{i}}|d\varphi|^{p(x)-2} \nabla^{\varphi}_{e_{i}}( v\circ\varphi )
& =&\nonumber -(\lambda \circ\varphi)d\varphi(\operatorname{grad}^{M}|d\varphi|^{p(x)-2})\\
&&- \sum_{i=1}^m|d\varphi|^{p(x)-2}<d\varphi(e_{i}),v\circ\varphi>d\varphi(e_{i}) \nonumber\\
&& -(\lambda \circ\varphi)|d\varphi|^{p(x)-2}\tau(\varphi).
\end{eqnarray}
By the  $p(\cdot)$-harmonicity condition of $\varphi$
$$\tau_{p(\cdot)}(\varphi)=|d\varphi|^{p(x)-2}\tau(\varphi)+d\varphi(\operatorname{grad}^M|d\varphi|^{p(x)-2})=0,$$
and equation \eqref{32}, we get
\begin{eqnarray}\label{33}
\sum_{i=1}^m\langle\nabla^{\varphi}_{e_{i}}|d\varphi|^{p(x)-2} \nabla^{\varphi}_{e_{i}}(v\circ\varphi ), v\circ\varphi\rangle
& = &\nonumber - \sum_{i=1}^m |d\varphi|^{p(x)-2}\langle d\varphi(e_{i}),v\circ\varphi\rangle^2.\\
\end{eqnarray}
Since the sphere $\mathbb{S}^{n}$ has constant curvature, we have
\begin{eqnarray}\label{34}
\sum_{i=1}^m\langle|d\varphi|^{p(x)-2}R^{\mathbb{S}^{n}}(v\circ\varphi, d\varphi(e_{i}))d\varphi(e_{i}),v\circ\varphi\rangle
= |d\varphi|^{p(x)}\langle v\circ\varphi, v\circ\varphi\rangle \nonumber   \\
-\sum_{i=1}^m|d\varphi|^{p(x)-2}\langle d\varphi(e_{i}), v\circ\varphi\rangle^2.\qquad
\end{eqnarray}
By the definition of generalized Jacobi operator, and  \eqref{33}, \eqref{34}, we obtain
\begin{eqnarray}\label{35}
\langle J_{f}^{\varphi}(v\circ\varphi), v\circ\varphi\rangle
& =&\nonumber  2|d\varphi|^{p(x)-2}\sum_{i=1}^m\langle d\varphi(e_{i}),v\circ\varphi\rangle^2\\
&&\nonumber - |d\varphi|^{p(x)}\langle v\circ\varphi, v\circ\varphi\rangle   \\
&&\nonumber-\sum_{i=1}^m\langle\nabla^\varphi_{e_i}(p(x)-2)|d\varphi|^{p(x)-4}\langle\nabla^\varphi v\circ\varphi  ,d\varphi\rangle d\varphi(e_i),v\circ\varphi\rangle,\\
\end{eqnarray}
Using $\langle\nabla^{\varphi}v\circ\varphi, d\varphi\rangle=-(\lambda\circ\varphi)|d\varphi|^2$, and equation (\ref{35}), we find that
\begin{eqnarray}\label{36}
\operatorname{trace}_\alpha  \langle J_{f}^{\varphi}(v\circ\varphi), v\circ\varphi\rangle
& = & (p(x)-n)|d\varphi|^{p(x)}.
\end{eqnarray}
Hence Theorem \ref{th2:sphere} follows from \eqref{36}, and the stable $p(\cdot)$-harmonicity condition of $\varphi$,
with $2\leq p(x)<n$ for all $x\in M$.
\end{proof}


\section{$p(\cdot)$-Biharmonic Maps}
Let $\varphi:(M,g)\longrightarrow (N,h)$ be a smooth map between two Riemannian manifolds, the $p(\cdot)$-bienergy of $\varphi$ is defined by
\begin{equation}\label{eq4.1}
E_{2,p(\cdot)}(\varphi;D)=\frac{1}{2}\int_{D}|\tau_{p(\cdot)}(\varphi)|^{2}v_g,
\end{equation}
where $p\geq2$ is a smooth function on $M$, and $D$ a compact subset of $M$. A smooth map $\varphi$ is called $p(\cdot)$-biharmonic if it is a critical point of the $p(\cdot)$-bienergy functional for any compact domain $D$.
\begin{theorem}[The first variation of the $p(\cdot)$-bienergy]\label{th6}
Let $\varphi$ be a smooth map between two Riemannian manifolds $(M,g)$ and $(N,h)$. Then we have
	\begin{equation}\label{eq4.2}
	\frac{d}{dt}E_{2,p(\cdot)}(\varphi_{t};D)\bigg|_{t=0}=-\int_{D}h(v,\tau_{2,p(\cdot)}(\varphi))v_g,
	\end{equation}
where $\{\varphi_{t}\}_{t\in(-\epsilon,\epsilon)}$ is a smooth variation of $\varphi$ supported in $D$,
$v=\frac{d\varphi_{t}}{dt}\big|_{t=0}$ denotes the variation vector field, and
$\tau_{2,p(\cdot)}(\varphi)$ the $p(\cdot)$-bitension field of $\varphi$ given by
	\begin{align}\label{eq4.3}
 \tau_{2,p(\cdot)}(\varphi)&\nonumber=-|d\varphi|^{p(x)-2}\operatorname{trace}_gR^{N}(\tau_{p(\cdot)}(\varphi),d\varphi)d\varphi
    -\operatorname{trace}_g\nabla^{\varphi}|d\varphi|^{p(x)-2}\nabla^{\varphi}\tau_{p(\cdot)}(\varphi)\\
	&\quad-\operatorname{trace}_g\nabla(p(x)-2)|d\varphi|^{p(x)-4}\langle\nabla^{\varphi}\tau_{p(\cdot)}(\varphi),d\varphi\rangle d\varphi.
	\end{align}
\end{theorem}
\begin{proof}
		Define $\phi:M\times (-\epsilon,\epsilon)\longrightarrow N$ by $\phi(x,t)=\varphi_{t}(x)$. First, note that
\begin{equation}\label{eq4.4}
    \frac{d}{dt}E_{2,p(\cdot)}(\varphi_{t};D)
    =\int_{D}h(\nabla^{\phi}_{\frac{\partial}{\partial {t}}}\tau_{p(\cdot)}(\varphi_{t}),\tau_{p(\cdot)}(\varphi_{t}))\,v_{g}.
\end{equation}
Calculating in a normal frame at $x\in M$, we have
\begin{equation}\label{eq4.5}
    \nabla^{\phi}_{\frac{\partial}{\partial {t}}}\tau_{p(\cdot)}(\varphi_{t})
    =\sum_{i=1}^m\nabla^{\phi}_{\frac{\partial}{\partial {t}}}\nabla^{\phi}_{(e_i,0)}|d\varphi_t|^{p(x)-2}d\phi(e_i,0).
\end{equation}
From the definition of the curvature tensor of $(N,h)$, we obtain\\

$\displaystyle\sum_{i=1}^m\nabla^{\phi}_{\frac{\partial}{\partial {t}}}\nabla^{\phi}_{(e_i,0)}|d\varphi_t|^{p(x)-2}d\phi(e_i,0)$
\begin{eqnarray}\label{eq4.6}
&=&\nonumber|d\varphi_t|^{p(x)-2}\sum_{i=1}^mR^{N}(d\phi(\frac{\partial}{\partial {t}}),d\phi(e_i,0))d\phi(e_i,0)\\
&&+ \sum_{i=1}^m\nabla^{\phi}_{(e_i,0)}\nabla^{\phi}_{\frac{\partial}{\partial {t}}}|d\varphi_t|^{p(x)-2}d\phi(e_i,0).
\end{eqnarray}
By using the compatibility of $\nabla^{\phi}$ with $h$,  we find that\\

$\displaystyle \sum_{i=1}^mh(\nabla^{\phi}_{(e_i,0)}\nabla^{\phi}_{\frac{\partial}{\partial {t}}}|d\varphi_t|^{p(x)-2}d\phi(e_i,0),\tau_{p(\cdot)}(\varphi_{t}))$
\begin{eqnarray}\label{eq4.7}
   &=&\nonumber \sum_{i=1}^m(e_i,0)\left[h(\nabla^{\phi}_{\frac{\partial}{\partial {t}}}|d\varphi_t|^{p(x)-2}d\phi(e_i,0),\tau_{p(\cdot)}(\varphi_{t}))\right] \\
   & &- \sum_{i=1}^mh(\nabla^{\phi}_{\frac{\partial}{\partial {t}}}|d\varphi_t|^{p(x)-2}d\phi(e_i,0),\nabla^{\phi}_{(e_i,0)}\tau_{p(\cdot)}(\varphi_{t})).\quad\quad\quad
\end{eqnarray}
From the property $\nabla^\phi_X d\phi(Y)=\nabla^\phi_Y d\phi(X)+d\phi([X,Y])$, with $X=\frac{\partial}{\partial {t}}$ and $Y=|d\varphi_t|^{p(x)-2}(e_i,0)$, we get\\

$\nabla^{\phi}_{\frac{\partial}{\partial {t}}}|d\varphi_t|^{p(x)-2}d\phi(e_i,0)\Big|_{t=0}$
\begin{eqnarray}\label{eq4.8}
    &=&\nonumber|d\varphi|^{p(x)-2}\nabla^\varphi _{e_i} v\\
    &&+\sum_{j=1}^m(p(x)-2)|d\varphi|^{p(x)-4}h(\nabla^\varphi_{e_j} v,d\varphi(e_j))d\varphi(e_i),
\end{eqnarray}
for all $i=1,...,m$. Substituting (\ref{eq4.8}) in (\ref{eq4.7}), we have\\

$\displaystyle \sum_{i=1}^mh(\nabla^{\phi}_{(e_i,0)}\nabla^{\phi}_{\frac{\partial}{\partial {t}}}|d\varphi_t|^{p(x)-2}d\phi(e_i,0),\tau_{p(\cdot)}(\varphi_{t}))\Big|_{t=0}$
\begin{eqnarray}\label{eq4.9}
   &=&\nonumber \sum_{i=1}^me_ih(|d\varphi|^{p(x)-2}\nabla^\varphi _{e_i} v,\tau_{p(\cdot)}(\varphi))\\
   &&\nonumber +\sum_{i=1}^me_ih((p(x)-2)|d\varphi|^{p(x)-4}\langle\nabla^\varphi v,d\varphi\rangle d\varphi(e_i),\tau_{p(\cdot)}(\varphi)) \\
   &&\nonumber - \sum_{i=1}^me_ih( v,|d\varphi|^{p(x)-2}\nabla^{\varphi}_{e_i}\tau_{p(\cdot)}(\varphi))\\
   &&\nonumber + \sum_{i=1}^mh(v,\nabla^\varphi _{e_i}|d\varphi|^{p(x)-2}\nabla^{\varphi}_{e_i}\tau_{p(\cdot)}(\varphi))\\
   &&\nonumber - \sum_{j=1}^me_jh(v,(p(x)-2)|d\varphi|^{p(x)-4}\langle\nabla^\varphi \tau_{p(\cdot)}(\varphi),d\varphi\rangle d\varphi(e_j))\\
   && + \sum_{j=1}^mh(v,\nabla^{\varphi}_{e_j}(p(x)-2)|d\varphi|^{p(x)-4}\langle\nabla^\varphi \tau_{p(\cdot)}(\varphi),d\varphi\rangle d\varphi(e_j)).
\end{eqnarray}
Let $\eta_1,\eta_2,\eta_3,\eta_4\in\Gamma(T^*M)$ defined by
\begin{eqnarray*}
 \eta_1(X)  &=& h(|d\varphi|^{p(x)-2}\nabla^\varphi _{X} v,\tau_{p(\cdot)}(\varphi)); \\
 \eta_2(X)  &=& h((p(x)-2)|d\varphi|^{p(x)-4}\langle\nabla^\varphi v,d\varphi\rangle d\varphi(X),\tau_{p(\cdot)}(\varphi)); \\
 \eta_3(X)  &=& h( v,|d\varphi|^{p(x)-2}\nabla^{\varphi}_{X}\tau_{p(\cdot)}(\varphi)); \\
 \eta_4(X)  &=& h(v,(p(x)-2)|d\varphi|^{p(x)-4}\langle\nabla^\varphi \tau_{p(\cdot)}(\varphi),d\varphi\rangle d\varphi(X)).
\end{eqnarray*}
Finally, we calculate the divergence of $\eta_i$ $(i=1,...,4)$ and substituting in (\ref{eq4.9}).
The proof of Theorem \ref{th6} follows by (\ref{eq4.4})-(\ref{eq4.6}), (\ref{eq4.9}), and the divergence Theorem.		
\end{proof}
From Theorem \ref{th6}, we deduce:
\begin{corollary}\label{co7}
Let $\varphi:(M,g)\rightarrow (N,h)$ be a smooth map between Riemannian manifolds. Then, $\varphi$ is $p(\cdot)$-biharmonic if and only if
\begin{eqnarray*}
\tau_{2,p(\cdot)}(\varphi)
   &=& -|d\varphi|^{p(x)-2}\operatorname{trace}_gR^{N}(\tau_{p(\cdot)}(\varphi),d\varphi)d\varphi\\
    &&   -\operatorname{trace}_g\nabla^\varphi |d\varphi|^{p(x)-2} \nabla^\varphi \tau_{p(\cdot)}(\varphi)\\
    &&-\operatorname{trace}_g\nabla (p(x)-2)|d\varphi|^{p(x)-4}\langle\nabla^\varphi\tau_{p(\cdot)}(\varphi),d\varphi\rangle d\varphi=0.
\end{eqnarray*}
\end{corollary}

\begin{remark}
For any smooth map $\varphi:(M,g)\longrightarrow(N,h)$ between two Riemannian manifolds, we have
$$\tau_{2,p(\cdot)}(\varphi)=J^\varphi_{p(\cdot)}(\tau_{p(\cdot)}(\varphi)).$$
We can extract several examples of $p(\cdot)$-biharmonic non $p(\cdot)$-harmonic maps $\varphi:(M,g)\longrightarrow\mathbb{R}^n$ where the $p(\cdot)$-tension field
is parallel along $\varphi$, i.e., the components of $\tau_{p(\cdot)}(\varphi)$ are constants.
\end{remark}

\begin{example}
Let $M=\{(x,y,z)\in\mathbb{R}^3,\,\sqrt{x^2+y^2}>2\}$. The smooth map $\varphi:M\longrightarrow\mathbb{R}^2$ defined by
 $$\varphi(x,y,z)=(\sqrt{x^2+y^2},z),\quad\forall (x,y,z)\in M,$$ is $p(\cdot)$-biharmonic non $p(\cdot)$-harmonic, where $$p(x,y,z)=\frac{\ln(x^2+y^2)}{\ln(2)},$$ for all $(x,y,z)\in M$. Here, $\tau_{p(\cdot)}(\varphi)=(1,0)$.
\end{example}


\begin{thebibliography}{99}


\bibitem{baird} P. Baird and J. C. Wood, { Harmonic morphisms between Riemannain manifolds}, Clarendon Press, Oxford (2003).

\bibitem{BG} P. Baird,  S. Gudmundsson,  { $p$-Harmonic maps and minimal submanifolds}, Math. Ann. {\bf294} (1992), 611-624.

\bibitem{BI} B. Bojarski,  and T. Iwaniec,  { $p$-Harmonic equation and quasiregular mappings},
Partial differential equations (Warsaw, 1984), 25-38, Banach Center Publ., vol. {\bf19}. PWN, Warsaw, 1987.

\bibitem{CL} L-F. Cheung and P-F. Leung, Some results on stable $p$-harmonic maps, Glasgow Math. J. 36 (1994) 77-80.


\bibitem{ES} J. Eells and J. H. Sampson, { Harmonic mappings of Riemannian manifolds}, Amer. J. Math. {\bf 86} (1964), 109-160.

\bibitem{ali} A. Fardoun, { On equivariant $p$-harmonic maps}, Ann.Inst. Henri. Poincare, {\bf15} (1998), 25-72.

\bibitem{Jiang} G. Y. Jiang, { 2-Harmonic maps between Riemannian manifolds}, Annals of
Math., China, {\bf 7A(4)} (1986), 389-402.

\bibitem{cherif2}  A. Mohammed Cherif, { On the $p$-harmonic and $p$-biharmonic maps}, J. Geom. (2018) 109:41

\bibitem{NS} T. Nagano and M. Sumi, Stability of $p$-harmonic maps, Tokyo J. Math.
Vol. 15, No. 2, 1992.


\bibitem{YX} Y. Xin, { Geometry of harmonic maps}, Fudan University, 1996.

\end{thebibliography}
\end{document}